\documentclass[12pt,letterpaper,reqno]{amsart}

\usepackage{amsthm}
\usepackage{mathrsfs}
\usepackage{amsfonts,amssymb,amsmath}
\input amssym.def 
\input amssym.tex

\usepackage[dvips]{graphicx}
\usepackage{hyperref}

\newtheorem{lemma}{\bf Lemma}[section]
\newtheorem{prop}{\bf Proposition}[section]
\newtheorem{thm}{\bf Theorem}[section]
\newtheorem{corr}{\bf Corollary}[section]

\newcommand\gc{\gcd}
\newcommand\eps{\epsilon}

\newcommand\be{\begin{eqnarray*}}
\newcommand\ee{\end{eqnarray*}}
\newcommand\beq{\begin{equation}}
\newcommand\eeq{\end{equation}}

\newcommand\ben{\begin{eqnarray}}
\newcommand\een{\end{eqnarray}}
\newcommand\ord{\mathrm{ord}}

\newcommand{\cO}{\mathcal{O}}
\newcommand{\gp}{\mathfrak{p}}
\newcommand{\gB}{\mathfrak{B}}

\begin{document}

\title[Improved bounds for arithmetic progressions in product sets] 
{Improved bounds for arithmetic progressions in product sets}

\author{Dmitrii Zhelezov}
\thanks{Department of Mathematical Sciences, 
Chalmers University Of Technology \and
Department of Mathematical Sciences,  University of Gothenburg} 
\address{Department of Mathematical Sciences,
Chalmers University Of Technology, 41296 Gothenburg, Sweden} 
\address{Department of Mathematical Sciences, University of Gothenburg,
41296 Gothenburg, Sweden} \email{zhelezov@chalmers.se}


\date{\today}

\begin{abstract}
  Let $B$ be a set of natural numbers of size $n$. We prove that the length of the longest arithmetic progression contained in the product set $B.B = \{bb'| \, b, b' \in B\}$ cannot be greater than $O(n \log n)$ which matches the lower bound provided in an earlier paper up to a multiplicative constant. For sets of complex numbers we improve the bound to $O_\eps(n^{1 + \eps})$ for arbitrary $\eps > 0$ assuming the GRH.
\end{abstract}

\maketitle
\section{Introduction}
  
 In \cite{ZH1} we have studied a relationship between the additive structure and the size of a product set in the following sense. It was shown that a product set $B.B = \{bb' |b,b'\in B\}$ of a set of natural numbers $B$ of size $n$ cannot contain arithmetic progressions longer than $O(\frac{n\log^2 n}{\log \log n})$. In the present note want to improve this result and show that a better bound actually holds.
\begin{thm} \label{thm:main}
  Suppose that $B$ is a set of $n$ natural numbers. Then the longest arithmetic progression in $B.B$ has length at most $O(n \log n)$.
\end{thm} 

In the same paper \cite{ZH1} it was shown that there is a set $B$ of $n$ natural numbers such that $B.B$ contains an arithmetic progression of length $\Omega(n \log n)$, so the bound of Theorem \ref{thm:main} is tight up to a multiplicative constant. 

In the second part of the paper we generalize Theorem \ref{thm:main} to get a slightly weaker bound 
for complex numbers. Unfortunately, we need certain sharp estimates (namely, an effective version of the Chebotarev density theorem) so the result is conditional on the Generalized Riemann Hypothesis (GRH).

\begin{thm} \label{thm:complex}
  Suppose that $B$ is a set of $n$ complex numbers. Then the longest arithmetic progression in $B.B$ has length at most $O_\eps(n^{1+\eps})$ for any $\eps > 0$, assuming GRH.
\end{thm} 

So the problem to unconditionally improve the bound $O(n^{3/2})$ obtained in \cite{ZH1} for sets of complex numbers remains open. It is interesting to note that the polynomial congruences arising in the proof of Theorem \ref{thm:complex} were studied by Bourgain, Garaev, Konyagin and Shparlinski \cite{BGKS} in a slightly different context. 


\medskip
\noindent {\bf Remark.} It is possible to track down the explicit dependency on $\eps$, so in fact Theorem \ref{thm:complex} holds with the absolute bound
$$
	|AP| = O\left(n^{1+\frac{\log \log n}{\log^{1/2} n}}\right).
$$

\section{Notation}
The following notation will be used in this paper: 
\begin{enumerate}
\item $f(n) = O(g(n))$ means that $\limsup_{n \rightarrow \infty} \frac{f(n)}{g(n)} < \infty$.
\item $f(n) = \Omega(g(n))$ means that $g(n) = O(f(n))$.
\item $X \ll Y$, resp. $X \gg Y$ means $X = O(Y)$, resp. $Y = O(X)$. If not stated explicitly, when using such notation, we assume that $X$ and $Y$ depend on a large parameter $n$, which is normally the size of the set in question. For example, for two sets $A$ and $B$, $|A| \gg |B|$ means that $|B|/|A|$ is bounded from above by a constant independent of the sizes of $A$ and $B$. 
\item When dealing with arithmetic progressions, we will often write $\{a + d[N] \}$ as a shortcut for the set
$$
	\{a + di\,;\,\, i = 0, \ldots, N-1 \}.
$$
Note, that $N$ is then the size of the set. Likewise, $\{c(a + d[N])\}$ is just $\{ca + cd[N]\}$ and for an arbitrary function $f(\cdot)$,
$$
	\{ f([N]) \} = \{f(i)\,;\,\, i = 0, \ldots, N-1 \}.
$$

\item Let $H$ be a fixed graph. Then $\mathrm{ex}(n, H)$ denotes 
the maximal number of edges among all graphs with $n$ vertices which does not contain $H$ as a subgraph. In particular, $\mathrm{ex}(n, C_{k})$ denotes the maximal number of edges a graph with $n$ vertices which avoids cycles of length $k$.
\item Let $p$ be a prime, then $d = \ord_p(n)$ denotes the maximal power of $p$ such that $p^d \mid n$.
\item For a natural number $n$, $\omega(n)$ denotes the number of distinct prime factors of $n$. $\omega(n, k)$ counts only the prime factors $p \geq k$.
\end{enumerate}

\section{Integer arithmetic progressions} 

	Let $A$ be a set of natural numbers of size $N$. One of the obstructions which prevents $A$ from being a subset of a product set of considerably smaller size (say of order $\sqrt{N}$) is a large number of distinct prime factors of its elements. In particular, this is the case if $A$ is an arithmetic progression and if the size of the elements of $A$ is bounded in terms of $N$. If $A$ is an arithmetic progression, it is convenient to assume that $A$ is of the form 
\beq \label{eq:factor}
A = \{ D(r + d[N]) \},
\eeq
with $\gc(Dr, d) = 1$, which one can do without loss of generality (by Lemma 1 of \cite{ZH1}). The exact claim is then as follows (see Lemma 4 of \cite{ZH1}) . 
	
\begin{lemma} \label{lem:main}
  Let $\gc(Dr, d) = 1$ and $A = \{D(r + d[N])\}$ be an arithmetic progression contained in a product set $B.B$ with $|B| = n$. If $d < N^{k}$, $r < N^{k}$, $k = o(\log N)$ then $N \leq 36kn \log n$ for sufficiently large $n$.
\end{lemma}
	
We are going to prove that either the hypothesis of Lemma \ref{lem:main} is fulfilled or there are many distinct prime divisors of the elements of $A$.

The following lemma is a slight modification of an inequality due to Erd\H{o}s as cited in \cite{M}. The proof is elementary yet illuminating. 

\begin{lemma} \label{lem:ErdosIneq}
	
	Let $\gc(d,r) = 1$ and $\Pi = r(r+d)\ldots(r+(N-1)d)$. For every prime factor $p$ of $\Pi$, leave out a term $r + id$ with the maximal $\ord_p$, that is, for which $p^m | r + id$ and $p^{m+1}$ divides none of the elements (resolving ties arbitrarily and obviously a term is removed at most once). Let $M$ be the total number of elements that are left out. Then
	
	\beq \label{eq:erdosineq}
		r(r+d)\ldots(r + (N-1-M)d) \leq (N-1)!.
	\eeq
\end{lemma}
\begin{proof}
  After the elimination process described in the statement of the lemma is finished, the remaining terms have only prime factors less than $N$ coprime with $d$. The order of such a prime $p$ is at most 
  $$
  	\sum^\infty_{i=1} \left[ \frac{N-1}{p^i}\right],
  $$
	that is, the power of $p$ in $(N-1)!$. Therefore, the product of the remaining terms divides $(N-1)!$. The claim of the lemma then follows as the left hand side is certainly less than or equal to the product of the remaining terms. 
	
\end{proof}

\begin{corr} \label{corr:ineq}
	Either both $d, r < N^2$ or $M > N/2$ assuming $N$ is sufficiently large.
\end{corr}
\begin{proof}
	Follows from (\ref{eq:erdosineq}).
\end{proof}

Before we proceed with the proof, let us define an auxiliary bipartite graph $G(A, B.B)$ which in general can be constructed for any sets $A$ and $B$ whenever $A \subset B.B$. The color classes of $G$ are just two copies of $B$ and for each $a \in A$ we pick a unique representation $a = b_1b_2$ (for example, take the lexicographically smallest pair) and place an edge $(b_1, b_2)$ in $G$. Record that $V(G) = 2|B|$ and $E(G) = |A|$. We will call this graph the \textit{containment} graph of $A$ in $B.B$.

\begin{proof}\textbf{[of Theorem \ref{thm:main}]}

		By Corollary \ref{corr:ineq} and Lemma \ref{lem:main}, it suffices to handle the case $M > N/2$. Let $p_1, ..., p_M$ be the primes and $a_1, ..., a_M$ the corresponding elements of $A/D = \{r + d[N]\}$, left out by the process described in Lemma \ref{lem:ErdosIneq}, applied to the product
$$
	\Pi = \prod_{a \in A/D}a = \prod_{j=0}^{N-1} (r + jd).
$$		 
		 
		 First, we prune from the set $A' = \{a_1, \ldots, a_M\}$ all elements $a_i$ with $p_i < N/2$. By the Prime Number Theorem, $o(N)$ elements are pruned, so with a slight abuse of notation we may assume that $M > N/3$ and $p_i > N/2$. Then, each $p_i$ divides at most two elements, so we can pick a subset of primes $P = \{ p_i \}$ and a subset $A'' \subset A'$ with $|A''| > N/6$, such that each $p \in P$ divides exactly one element in $A''$. Note, that $A''$ is as well a subset of $A/D$, and what we are going to show is that $DA'' \subset B.B$ implies $|B| = \Omega(N)$. Clearly, this implies the claim of the theorem. 
		
		Let $G(DA'', B.B)$ be the containment graph of $DA''$ in $B.B$.	We claim that $G$ contains no cycles. Suppose not and $C = \langle b_1....b_{2m} \rangle$ is a cycle in $G$ ($C$ is even since $G$ is bipartite). By construction, there is a prime $p$ which divides $b_1b_2/D$ and does not divide $b_ib_{i+1}/D$ for $i > 1$ (indices are taken modulo $m$). In other words, $\ord_p(b_1b_2) > \ord_p(D)$ and $\ord_p(b_ib_{i+1}) = \ord_p(D)$ for $i > 1$. On the other hand one can write
		$$
			\prod_{\text{ $i$ is even}} b_{i}b_{i+1} = \prod_{\text{ $i$ is odd}}  b_{i}b_{i+1},
		$$ 
	which leads to a contradiction, as the order of $p$ on the left hand side is greater than on the right hand side. Thus, $|A''| = E(G) < V(G) = 2|B|$ and we are done.
		
\end{proof}

Following the same route, we can state the following general proposition.

\begin{prop} \label{prop:general}
	Let $A$ be a set of integers of size $N$ such that for every $a \in A$ holds the bound $a/\gc(A) < N^{K_1}$, where $\gc(A)$ denotes the GCD of all the elements in $A$. Suppose that $A \in B.B$ for some set $B$ and 
	$$
	\omega\left(\prod_{a \in A}\frac {a}{\gc(A)}, K_2N\right) \geq M.
	$$ 
	Then $|B| = \Omega_{K_1, K_2}(M)$. 
\end{prop}
\begin{proof}
	Because of the bound $a/\gc(A) < N^{K_1}$ each element of $\{a/\gc(A): \,\, a \in A \}$ can have at most $O(1)$ distinct prime factors which are greater than $K_2N$. Thus, we can choose $\Omega(M)$ primes $P = \{p_i\}$ and $A' = \{ a_i \}$ such that $p_i > K_2N$ and each $p_i$ divides exactly one element of $A'/\gc(A)$. Applying verbatim the argument of the proof of Theorem \ref{thm:main} we then conclude that $|B| = \Omega(|A'|) = \Omega_{K_1, K_2}(M)$
\end{proof}
  
\section{The case of complex numbers}  
 
 We now switch to the complex case. In \cite{ZH1} it was shown that if $B$ is a set of complex numbers and $B.B$ contains an AP $A$ with $|A| = \Omega\left(|B|^{3/2}\right)$, one can construct a set $B'$ of integers such that $B'.B'$ contains an AP of the same length. Though we don't see how this argument extends to shorter sequences, it is still possible to deduce some structural information from the fact that the product set contains a long AP. 

As we are dealing with complex numbers now, it is safe to assume that $A = \{r + [N] \}$, rescaling each element of $B$ by a square root of the difference if necessary. Let $G(A, B.B)$ be the containment graph of $A$ in $B.B$. Then, if we consider $r$ as a variable, any cycle in $G$ implies a polynomial relation in $r$ with integer coefficients, so in fact $r$ must be algebraic, and the degree is controlled by the cycle length. 

Before we proceed, recall that the \emph{height} of an integer polynomial $P$ is defined as the largest (by absolute value) coefficient of $P$. It will also convenient for our purposes to define the height of an algebraic number $\alpha$ as the height of its primitive polynomial, that is, the integer polynomial of minimal degree and minimal in absolute value leading coefficient, of which $\alpha$ is a root. 

\begin{lemma} \label{lemma:cycles}
	If $G(A, B.B)$ contains a $2k$ cycle then $r$ is an algebraic number of degree at most $k$ and height at most $O(N^k)$.
\end{lemma}
\begin{proof}
	Let $\langle b_0b_1\ldots b_{2k-1}\rangle$ be a cycle in $G$. Then we have
	$$
		\prod_{i \text{ is even}} b_{i}b_{i+1} = \prod_{i \text{ is odd}}  b_{i}b_{i+1},
	$$
	or after the substitution $b_ib_{i+1} = a_{j(i)} = r + j$ for some $j(i)$ we have a polynomial equation in $r$ with integer coefficients
	$$
		\prod_{i \text{ is even}} (r+j(i)) = \prod_{i \text{ is odd}} (r + j(i)).
	$$
	The leading term cancels out and we arrive at an integer polynomial of degree less than $k$. The polynomial is not identically zero since  the $j$s are distinct. Moreover, the coefficients of the polynomial are bounded by $O(N^k)$. Indeed, for $0 \leq l < k$ the coefficient in the term $r^l$ is a sum of $(k-l)$-fold products of $j$s, and each $j$ is in turn bounded by $N$. The number of summands is bounded by $2 \binom {k}{l}$, thus each coefficient is bounded by $O(N^k)$ as $2k \leq V(G) \leq N$. 
\end{proof}  

Let $K = \mathbb{Q}(r)$ be the field extension with $r$ adjoined. The following lemma, which is a straightforward generalization of \cite{ZH1} Lemma 7, allows us to assume that in fact all the elements we are dealing with lie in $K$.

\begin{lemma} \label{lemma:kreduction}
	Let $A \subset B.B$ and $A \subset K$ for some field $K$. Then there is $B'$ with $|B'| \leq 2|B|$ s.t. $B' \subset K$ and $A \subset B'B'$. 
\end{lemma}

We will need some estimates from number theory to estimate the number $\pi(P, x)$ of primes $p < x$ such that $P(\cdot) = 0$ has a solution modulo $p$, where $P \in \mathbb{Z}[x]$, that is, $P$ has a linear factor in $\mathbb{F}_p$. Here we record a lemma we will use, which is a consequence of a certain effective version of the Chebotarev Density Theorem. The original statement, proved by Lagarias and Odlyzko, is somewhat technical, so we defer it to the Appendix.

\begin{lemma} \label{lemma:algprimes}
  Assume $P$ is an irreducible integer polynomial of degree $d$ with coefficients bounded by $O(N^k)$. Then, assuming GRH,
  \beq \label{eq:boundprimes}
  	\pi(P, N) = \Omega_{d,k}\left(\frac{N}{\log N}\right).
  \eeq
\end{lemma}


To close the argument, we now prove the size of $B$ is actually controlled by the degree and height of $r$, which would imply that either the containment graph does not contain small even cycles and thus sparse, or $B$ is not very large.

\begin{lemma} \label{lem:bounds}
	Suppose $A = \{r + [N]\} \subset B.B$. Let $k$ be fixed and $r$ is an algebraic number of degree $k$ and height $O(N^k)$. Then, assuming GRH,  $|B| = \Omega_{d,k}(N/\log N)$.
\end{lemma}

\begin{proof}

 Let $K = \mathbb{Q}(r)/\mathbb{Q}$, so $k = [\mathbb{Q}(r) : \mathbb{Q}]$. By Lemma \ref{lemma:kreduction}, there is a set $B'$ of algebraic numbers in $K$ of degree at most $k$ such that $B'.B'$ contains $\{ r + [N] \}$ and  $|B| \gg |B'|$. 
  
First, we want to make all elements of $B'$ algebraic integers. Let $M$ be the LCM of the denominators of $b$s in $B'$ (by the denominator of an algebraic number $\alpha$ we mean the minimal positive integer $D$ s.t. $D\alpha$ is integer in $\mathbb{Q}(r)$). Multiplying  every element by $M$ we get a new set $B'$ of algebraic integers (we use the same letter to shorten the notation) such that $B'.B'$ contains 
$$
A = \{ M^2(r + [N]) \}.
$$ 

Next, since $r$ is of height $O(N^k)$, the denominator of $r$ is bounded by $O(N^k)$, so there is a rational integer $m = O(N^k)$ s.t. 
  $mr$ is an integer in $\mathbb{Q}(r)$ of height $O(N^{k^2})$.  Denote $r' = rm$ and take $B_2 = mB' \cup B'$, so 
  $$
  A_2 = \{ M^2(r' + m[N]) \} \subset B_2.B_2. 
  $$  
  Now we have constructed a new set $B_2$ of  algebraic integers such that $B_2.B_2$ contains an AP 
  $$
  A_2 = \{ M^2(r' + m[N]) \}
  $$ with $m = O(N^k)$ and $r'$ is an algebraic integer of height $O(N^{k^2})$.

We claim that the norm $N_{\mathbb{Q}(r)/\mathbb{Q}}(r' + mi)$ is an integer polynomial in $i$ of degree $k$ and coefficients bounded by $O(N^{k^2})$. Indeed, let $\sigma_j, j = 1 \ldots k$ be the distinct $\mathbb{C}$-embeddings of $\mathbb{Q}(r)$. We have
$$
N(r' + mi) = m^k N(r + i) = m^k \prod_{j=1}^k (\sigma_j(r) + i) = (-1)^k P(-i) m^k,
$$
where $P$ is the minimal polynomial of $r$. On the other hand, by our choice of $m$,  $m^k P(r)$ is the minimal polynomial of $mr$ (since it is an integer monic polynomial with respect to $mr$ of degree $k$) which in turn has integer coefficients bounded by $O(N^{k^2})$.
  	
The final step is to take norms. Let 
$$
B_3 = \{N(b)\,|\, b \in B_2\}
$$ and 
$$
A_3 = \{N(a)\,|\, a \in A_2 \}.
$$ Now $B_3$ and $A_3$ consist of rational integers and
$$
 A_3 = \{ M^{2d} f([N]) \} \subset B_3.B_3,
$$ where  $f$ is a polynomial of degree $k$ and height $O(N^{k^2})$. We then apply Lemma \ref{lemma:algprimes} to deduce that there are at least $C_{k} N/\log N$ distinct primes $p < N$ such that $f$ has a root modulo $p$.  In other words, if $p < N$ is such a prime then there is an index $i$  s.t. $f(i) \equiv 0 \text{ mod } p$. By the Prime Number Theorem, at least half of such primes are greater than $C_{k} N/4$ for large $N$, so 
\beq \label{eq:omegaf}
\omega \left( \prod_{i=0}^{N-1}f(i), \frac{C_{k} }{4}N \right) \geq \frac{C_k N}{2 \log N}.
\eeq

On the other hand, $f(i) = O_k(N^{k^2+1})$ since the coefficients are bounded by $O(N^{k^2})$ and the degree of $f$ is $k$. In particular, certainly $\gc(f[N]) = O_k(N^{k^2+1})$ which together with (\ref{eq:omegaf}) gives
$$
\omega \left( \prod_{a \in A_3}\frac{a}{\gc(A_3)},  \frac{C_k }{4}N \right) = \omega \left( \prod_{i=0}^{N-1}\frac{f(i)}{\gc(f[N])}, \frac{C_{k} }{4}N \right) = \Omega_k\left(\frac{N}{ \log N}\right).
$$

Now everything is ready to apply Proposition \ref{prop:general} to the containment $A_3 \subset B_3.B_3$ and conclude that $|B_3| = \Omega_{k}(N/\log N)$. Since $|B_3| \ll |B|$, the claim follows.

\end{proof}

It remains to apply a well known result which bounds the number of edges in a $2k$-cycles free graph.

\begin{proof} \textbf{[of Theorem \ref{thm:complex}]}

Fix $\eps > 0$ and let $k = \lfloor \frac{1}{\eps} \rfloor$. If $G(A, B.B)$ does not contain $2k$-cycles, then by the bound of Bondy and Simonovits \cite{BS},
$$
	|A| = E(G) = O_k(|B|^{1 + 1/k}),
$$ 
so we are done. 

Otherwise, by Lemma \ref{lemma:cycles}, $r$ is an algebraic number of degree at most $k$ and height at most $O(N^k)$. Then Lemmas \ref{lemma:cycles} and \ref{lem:bounds} imply $|A| = O_k(|B|^{1+o(1)})$.

\end{proof}

\section{Acknowledgements}
I am grateful to the math community supporting MathOverflow and in particular to the users Joel, Lucia and Aurel for helpful comments and explanations. I also thank Yuri Bilu and Peter Hegarty for attention to this work. Last, but not least, I would like to thank an anonymous referee for comments and suggestions.

\section*{Appendix}

Let $A$ be a Dedekind domain, $K$ its quotient field, $L$ a finite extension of $K$ with $\mathrm{Gal}(L/K) = G$. Let also $\mathfrak{p}$ be a prime ideal in $A$ that does not ramify in $L$ and $\mathfrak{B}$ be an ideal of $B = \mathcal{O}_L$ lying over $\mathfrak{p}$. Recall that if $L/K$ is abelian, the Frobenius element (see e.g. \cite{N}) is a unique element $\sigma_\mathfrak{p} \in G$ such that for all $\alpha \in \mathcal{O}_L$ holds
$$
	\sigma_\mathfrak{p}(\alpha) = \alpha^{\#\mathcal{O}_K/\mathfrak{p}} \,\,\,\, ( \text{mod } \mathfrak{p}\mathcal{O}_L).
$$

For a non-Abelian extension $L/K$, the Frobenius symbol depends on the choice $\mathfrak{B}$ lying over $\mathfrak{p}$, and we have for all $\alpha \in \cO_L$
$$
	\sigma_{\gB}(\alpha) = \alpha^{\#\mathcal{O}_K/\mathfrak{p}} \,\,\,\, ( \text{mod } \gB).
$$

In this case all such elements $\sigma_{\gB}$ are conjugated, and the Frobenius symbol $[\frac{L/K}{\gp}]$ is defined as the corresponding conjugacy class of $G$. 

The crucial property of the Frobenius symbol is that it determines the \emph{decomposition type} of $\gp$ in $\cO_L$. In particular, if $L$ is the splitting field of a polynomial $f$ and $\gp$ does not ramify (i.e. does not divide the discriminant of $f$), the cycle structure of $\sigma_\gp$, viewed as a permutation of the roots of $f$, corresponds to the degrees of the irreducible factors of $f$ modulo $\gp$ in $(\cO_K/\gp)[X]$. In particular, $f$ has a root mod $\gp$ if and only if $\sigma_\gp$ fixes at least one root, which corresponds to a linear factor. Thus, to estimate the number of such primes, it suffices to estimate the number of prime ideals $\gp$ such that the corresponding conjugacy class contains at least one 1-cycle. For simplicity, we can take the conjugacy class containing only the identical permutation, i.e. take only prime ideals such that $f$ factors completely into linear factors mod $\gp$. The desired estimate is then given by the Chebotarev Density Theorem, and for our purposes we need the quantitative effective version proved by Lagarias and Odlyzsko under GRH.

\begin{thm}[Lagarias and Odlyzko, \cite{LO}]
Let $L/K$ be a Galois extension of number fields, $d_L$ the absolute discriminant of $L$, $n_L$ the degree of $L$ over $\mathbb{Q}$. Let $G$ be the Galois group of this extension and $C$ is a conjugacy class of $G$. For all $x>1$, let $\pi_C(x)$ denote the number of prime ideals $\mathfrak{p}$ of $K$ of norm less than or equal to $x$, that do not ramify in $L$ and such that 
$\left[ \frac{L/K}{\mathfrak{p}} \right] = C$. Then, assuming GRH,
$$
	 | \pi_C(x) - \frac{|C|}{|G|}\mathrm{Li}(x) | \leq c\left(\frac{|C|}{|G|}\sqrt{x} \log (d_Lx^{|G|}) + |C| \log(d_L)\right),
$$
where $c$ is an absolute constant.
\end{thm}

To deduce Lemma \ref{lemma:algprimes}, it remains to take $L$ the splitting field of our polynomial $P$, $ C = \{ \mathrm{id} \}$ (which is in fact is more restrictive) and note that the discriminant is bounded by $N^{O_{d, k}(1)}$, while the order of the Galois group is bounded by $d! = O_d(1)$.


\begin{thebibliography}{99}
	\bibitem{BeBi} D. Berend and Y. Bilu, \textit{Polynomials with roots modulo every integer}, Proc. Amer. Soc. 124 (6) (1996), 1663--1671.
		\bibitem{BS} J. A. Bondy and M. Simonovits, \textit{Cycles of even length in graphs}, J. Combinatorial Theory Ser. B, 16 (1974), 97--105.
	  \bibitem{BGKS} J. Bourgain, M. Garaev, S. Konyagin and I. Shparlinski, \textit{On congruences with products of variables from short intervals and applications}, Proceedings of the Steklov Institute of Mathematics 280 (1) (2013), 61--90.
	
		\bibitem{C} M.-C. Chang, \textit{Factorization in generalized arithmetic progressions and applications to the
Erd\H{o}s-Szemer\'edi sum-product problems}, Geom. Funct. Anal. Vol. 13 (2003), 720--736.

	\bibitem{M} P. Moree, \textit{On arithmetic progressions having only few different prime factors in comparison with their length}, Acta Arithm. LXX.4, (1995).
	\bibitem{N} W. Narkiewicz, \textit{Elementary and Analytic Theory of Algebraic Numbers}, 2nd ed., Springer-Verlag, and Polish Scientific Publishers, Warsaw, (1990). 
	\bibitem{LO} J. C. Lagarias and A. M. Odlyzko, \textit{Effective versions of the Chebotarev density theorem}, Algebraic Number Fields, A. Frohlich (ed.), Academic Press (1977), 409--464.
	
	\bibitem{ZH1}  D. Zhelezov, \textit{Product sets cannot contain long arithmetic progressions}, Acta Arith. 163 (2014), 299-307.
	
\end{thebibliography}
\end{document}